\documentclass[preprint,11pt]{imsart}
\pdfoutput=1
\RequirePackage[OT1]{fontenc}
\usepackage{amsthm,amsmath,amssymb,natbib,bm,enumitem}
\RequirePackage[colorlinks,citecolor=blue,urlcolor=blue]{hyperref}
\usepackage{maruyama}

\startlocaldefs
\numberwithin{equation}{section}
\theoremstyle{plain}
\newtheorem{thm}{Theorem}[section]
\newtheorem{lemma}{Lemma}[section]

\theoremstyle{definition}

\theoremstyle{remark}
\newtheorem{remark}{Remark}[section]
\endlocaldefs
\def\citeapos#1{\citeauthor{#1}'s (\citeyear{#1})}
\def\T{{ \mathrm{\scriptscriptstyle T} }}
\allowdisplaybreaks
\begin{document}

\begin{frontmatter}
 \title{A Gaussian sequence approach for proving minimaxity: A Review}
\runtitle{Minimaxity}

\begin{aug}
\author{\fnms{Yuzo} \snm{Maruyama}\thanksref{t1,m1}
\ead[label=e1]
{maruyama@mi.u-tokyo.ac.jp}}
 \and
\author{\fnms{William, E.} \snm{Strawderman}
\thanksref{t2,m2}
\ead[label=e2]{straw@stat.rutgers.edu}}

\thankstext{t1}{This work was partially supported by KAKENHI \#25330035, \#16K00040.}
\thankstext{t2}{This work was partially supported by grant from the Simons Foundation (\#418098 to William Strawderman).}
\address{University of Tokyo\thanksmark{m1} and Rutgers University\thanksmark{m2} \\
\printead{e1,e2}}
\runauthor{Y. Maruyama and W. Strawderman}

\end{aug}

\begin{abstract}
 This paper reviews minimax best equivariant estimation in these invariant estimation
 problems: a location parameter, a scale parameter and a (Wishart) covariance matrix.
 We briefly review development of the best equivariant estimator as a generalized Bayes
 estimator relative to right invariant Haar measure in each case.
 Then we prove minimaxity of the best equivariant procedure by giving a least favorable
 prior sequence based on non-truncated Gaussian distributions. The results in this paper
 are all known, but we bring a fresh and somewhat unified approach by using,
 in contrast to most proofs in the literature, a smooth sequence of non truncated priors.
 This approach leads to some simplifications in the minimaxity proofs.
\end{abstract}


\end{frontmatter}
\section{Introduction}
\label{sec:intro}
We review some results on minimaxity of best equivariant estimators from
what we hope is a fresh and somewhat unified perspective.
Our basic approach is to start with a general equivariant estimator, and
demonstrate that the best equivariant estimator is a generalized Bayes estimator, $\delta_0$,
with respect to an invariant prior.
We then choose an appropriate sequence of Gaussian priors whose support is the entirety of the
parameter space and show that the Bayes risks converge to the constant risk of
$\delta_0$.
This implies that $\delta_0$ is minimax.
All results on best equivariance and minimaxity, which we consider in this paper,
are known in the literature.
But, using a sequence of Gaussian priors as a least favorable sequence,
simplifies the proofs and gives fresh and unified perspective.

In this paper, we consider the following three estimation problems.
\begin{description}
 \item[Estimation of a location parameter:] 
	    Let the density function of $\bmX$ be given by
\begin{align}\label{location.density}
 f(\bmx-\mu)=f(x_1-\mu,\dots,x_n-\mu).
\end{align}
Consider estimation of the location parameter $\mu$ under location invariant loss
\begin{align}\label{general.loss.1}
L(\delta-\mu).
\end{align}
	    We study equivariant estimators under the location group,  given by
\begin{align}\label{location.equiv.estimator}
 \delta(\bmx-\mu)=\delta(\bmx)-\mu.
\end{align}
 \item[Estimation of a scale parameter:] 
	    Let the density function of $\bmX$ be given by
\begin{align}\label{scale.density}
 \sigma^{-n}f(\bmx/\sigma),
\end{align}
with scale parameter $\sigma$, where $\bmx/\sigma=(x_1/\sigma,\dots,x_n/\sigma)$.
Consider estimation of the scale $\sigma$ under scale invariant loss
\begin{align}\label{general.loss.2}
 L(\delta/\sigma).
\end{align}
We study equivariant estimators under scale group, given by
\begin{align}
 \delta(\bmx/\sigma)=\delta(\bmx)/\sigma.
\end{align}
\item[Estimation of covariance matrix:]	    
We study estimation of $\bmSi$ based on a $p\times p$ random matrix $\bmV$
	   having a Wishart distribution $\mathcal{W}_p(n,\bmSi)$,
	   where the density is given in \eqref{density.W} below.
An estimator $\bmde$ is evaluated by the invariant loss 
\begin{align}
 L(\bmSi^{-1}\bmde).
\end{align}
	   We consider equivariant estimators under the lower triangular group, given by
\begin{align}
 \bmde(\bmA\bmV\bmA^\T)=\bmA\bmde(\bmV)\bmA^\T,
\end{align}
where $ \mathcal{T}^+$, the set of $p\times p$ lower triangular matrices with
positive diagonal entries.
\end{description}
For the first two cases
with the squared error loss $(\delta-\mu)^2$ and the entropy loss
$\delta/\sigma-\log(\delta/\sigma)-1$, respectively,
the so called \cite{Pitman-1939} estimators
\begin{align}
  \hat{\mu}_0(\bmx)&=\frac{\int_{-\infty}^\infty \mu f(\bmx -\mu)\rd \mu}
 {\int_{-\infty}^\infty  f(\bmx -\mu)\rd \mu},\\
 \hat{\sigma}_0(\bmx)&=\frac{\int_0^\infty \sigma^{-n-1} f(\bmx/\sigma)\rd \sigma}
 {\int_0^\infty \sigma^{-n-2} f(\bmx/\sigma)\rd \sigma}
\end{align}
are well-known to be best equivariant and minimax.
Clearly, they are generalized Bayes with respect to $\pi(\mu)=1$ and $\pi(\sigma)=1/\sigma$,
respectively.
\cite{Girshick-Savage-1951} gave the original proof of minimaxity.
\cite{Kubokawa-2004} also gives a proof and further developments in the restricted parameter setting.
Both use a sequence of uniform distribution on expanding interval as least favorable priors.

For the last case, \cite{James-Stein-1961} show that the best equivariant estimator is given by
\begin{align}
 \hat{\bmSi}_0=\bmT\mathrm{diag}(d_1,\dots,d_p)\bmT^\T
\end{align}
where $\bmT\in\mathcal{T}^+$ is from the Cholesky decomposition of $\bmV=\bmT\bmT^\T$ and
$d_i=1/(n+p-2i+1)$ for $i=1,\dots,p$.
Note that the group of $p\times p$ lower triangular matrices with
positive diagonal entries is solvable, and the result of \cite{Kiefer-1957}
implies the minimaxity of $ \hat{\bmSi}_0$.
\cite{Tsukuma-Kubokawa-2015} gives as a sequence of least favorable priors,
the invariant prior truncated on a sequence of expanding sets.

In each case, the sequence of priors we employ is based on a Gaussian sequence
of possibly transformed parameters.
This is in contrast to most proofs in the literature which use truncated versions
of the invariant prior.
As a consequence, the resulting proofs are less complicated.

Section \ref{sec:GA} is devoted to developing the best equivariant estimator
as a generalized Bayes estimator with respect to a right invariant (Haar measure)
prior in each case. The general approach is basically that of \cite{Hora-Buehler-1966}.
Section \ref{sec:minimaxity} provides 
minimaxity proofs of the best equivariant procedure by giving a least favorable
prior sequence based on (possibly transformed) Gaussian priors in each cases.
We give some concluding remarks in Section \ref{sec:CR}.

\section{Establishing best equivariant procedures}
\label{sec:GA}
All results in this section are well-known.
Our proof of best equivariance for $\hat{\mu}_0$, $\hat{\sigma}_0$ and $\hat{\bmSi}_0$
follow from \cite{Hora-Buehler-1966}.
The reader is referred to 
\citeapos{Hora-Buehler-1966} for further details on their general development
of a best equivariant estimator as the generalized Bayes estimator relative to
right invariant Haar measure.
\subsection{Estimation of location parameter}
\label{sec:be_l}
Consider an equivariant estimator which satisfies $\delta(\bmx-\mu)=\delta(\bmx)-\mu$.
Then we have a following result.
\begin{thm}\label{thm:BE_location}
Let $\bmX$ have distribution \eqref{location.density} and let the loss be given by \eqref{general.loss.1}.
The generalized Bayes estimator with respect to the invariant prior $\pi(\mu)=1$, $\hat{\mu}_0(\bmx)$,
 is best equivariant under the location group, that is,
 \begin{align*}
 \hat{\mu}_0(\bmx)=\argmin_{\delta}
 \int_{-\infty}^\infty L(\delta(\bmx)-\mu)f(\bmx-\mu)\rd\mu.
\end{align*}
\end{thm}
\begin{proof}
The risk of the equivariant estimator \eqref{location.equiv.estimator}
is written as
\begin{align}
 &R(\delta(\bmx),\mu) \notag \\
 &=\int_{\mathbb{R}^n} L(\delta(\bmx)-\mu)f(\bmx-\mu)\rd\bmx \notag\\
 &=\int_{\mathbb{R}^n} L(\delta(\bmx-\mu))f(\bmx-\mu)\rd\bmx \notag\\
 &=\int_{\mathbb{R}^n} L(\delta(\bmz)) f(\bmz)\rd\bmz \label{risk.inv.location}\\
 &=\int_{\mathbb{R}^{n-1}}\int_{-\infty}^\infty
 L(\delta(\bmz_{n-1},z_n)) f(\bmz_{n-1},z_n)\rd z_n\rd\bmz_{n-1} \notag\\
 &=\int_{\mathbb{R}^{n-1}}\int_{-\infty}^\infty
 L(\delta(\bmz_{n-1},u_n-\theta)) f(\bmz_{n-1},u_n-\theta)\rd \theta\rd\bmz_{n-1} \notag\\
 &\qquad z_n=u_n-\theta \ (u_n\text{ is a constant and }\theta\text{ is variable}) \notag\\
 &=\int_{\mathbb{R}^{n-1}}\int_{-\infty}^\infty
 L(\delta(\bmu_{n-1}-\theta,u_n-\theta)) f(\bmu_{n-1}-\theta,u_n-\theta)\rd \theta\rd\bmu_{n-1} \notag\\
 &=\int_{\mathbb{R}^{n-1}}\left(
\int_{-\infty}^\infty
 L(\delta(\bmu)-\theta) f(\bmu-\theta)\rd \theta
 \right)\rd\bmu_{n-1}. \notag
\end{align}
Then the best equivariant estimator is 
\begin{align*}
 \hat{\mu}_0(\bmx)=\argmin_{\delta}
 \int_{-\infty}^\infty L(\delta(\bmx)-\mu)f(\bmx-\mu)\rd\mu.
\end{align*}
\end{proof}
\subsection{Estimation of scale}
\label{sec:be_s}
Consider an equivariant estimator which satisfies $\delta(\bmx/\sigma)=\delta(\bmx)/\sigma$.
Then we have a following result.
 \begin{thm}\label{thm:be_s}
Let $\bmX$ have distribution \eqref{scale.density} and let the loss be given by \eqref{general.loss.2}.
Then the generalized Bayes estimator, with respect to the prior $\pi(\sigma)=1/\sigma$, $\hat{\sigma}_0(\bmx)$, 
  is best equivariant under the scale group, that is,
\begin{align*}
 \hat{\sigma}_0(\bmx)=\argmin_{\delta}
 \int_0^\infty L(\delta/\sigma)\frac{f(\bmx/\sigma)}{\sigma^{n}}\frac{\rd\sigma }{\sigma}.
\end{align*}
 \end{thm}

\begin{proof}
The risk of the equivariant estimator is written as
\begin{align}
 &R(\delta(\bmx),\sigma) \notag \\
 &=\int_{\mathbb{R}^n} L(\delta(\bmx)/\sigma)\sigma^{-n}f(\bmx/\sigma)\rd\bmx \notag\\
 &=\int_{\mathbb{R}^n} L(\delta(\bmx/\sigma))\sigma^{-n}f(\bmx/\sigma)\rd\bmx \notag\\
 &=\int_{\mathbb{R}^n} L(\delta(\bmz)) f(\bmz)\rd\bmz \label{risk.inv}\\
 &=\int_{\mathbb{R}^{n-1}}\left(\int_{-\infty}^0+ \int_0^{\infty}\right)
 L(\delta(\bmz_{n-1},z_n)) f(\bmz_{n-1},z_n)\rd z_n\rd\bmz_{n-1} \notag\\
 &=\int_{\mathbb{R}^{n-1}}\sum_{j=\{-1,1\}} \int_0^{\infty}
 L(\delta(\bmz_{n-1},j z_n)) f(\bmz_{n-1},j z_n)\rd z_n\rd\bmz_{n-1} \notag\\
 &=\int_{\mathbb{R}^{n-1}}\sum_{j=\{-1,1\}} \int_0^{\infty}
 L(\delta(\bmz_{n-1},j u_n/w)) f(\bmz_{n-1},j u_n/w)\frac{u_n}{w^2}\rd w\rd\bmz_{n-1} \notag\\
 &\qquad z_n=u_n/w \ (\text{where }u_n\text{ is positive constant and }w\text{ is variable}) \notag\\
 &=\int_{\mathbb{R}^{n-1}}\sum_{j=\{-1,1\}} \int_0^\infty L(\delta(\bmu_{n-1}/w,ju_n/w)) \frac{1}{w^{n-1}}\frac{u_n}{w^2}
 f(\bmu_{n-1}/w,ju_n/w)\rd\bmu_{n-1}\rd w \notag\\   
 &\qquad z_i=u_i/w \ (i=1,\dots,n-1)\quad (\text{where }u_i\text{ is variable and }w\text{ is constant}) \notag\\
 &=\int_{\mathbb{R}^{n-1}}u_n\sum_{j=\{-1,1\}} \left\{\int_0^\infty L(\delta(\bmu_{n-1},ju_n)/w)
 \frac{f(\bmu_{n-1}/w,ju_n/w)}{w^{n+1}}\rd w\right\}
 \rd\bmu_{n-1}. \notag
\end{align}
Then the best equivariant estimator 
is 
\begin{align*}
 \hat{\sigma}_0(\bmx)=\argmin_{\delta}
 \int_0^\infty L(\delta/\sigma)\sigma^{-n-1}f(\bmx/\sigma)\rd\sigma .
\end{align*}
\end{proof}

\subsection{Estimation of covariance matrix}
\label{sec:be_l}
Let $\bmV$ have a Wishart distribution $\mathcal{W}_p(n,\bmSi)$. 
Let $\mathcal{T}^+$ be the set of $p\times p$ lower triangular matrices with
positive diagonal entries.
By the Cholesky decomposition, $\bmSi^{-1}$ and $\bmV$ can be written
as
\begin{align*}
 \bmSi^{-1}=\bmTh^\T \bmTh \text{ and }\bmV=\bmT\bmT^\T
\end{align*}
for $\bmTh=(\theta_{ij})\in\mathcal{T}^+$ and $\bmT=(t_{ij})\in\mathcal{T}^+$.
As in Theorem 7.2.1 of \cite{Anderson-2003},
the probability density function of $\bmT$ is
\begin{equation}\label{density.W}
 f_W(\bmT|\bmTh)\gamma(\rd \bmT)
=\frac{1}{C(p,n)}\left|\bmTh\bmT\right|^n\exp\left[-\frac{1}{2}\mathrm{tr}\left\{(\bmTh\bmT)(\bmTh\bmT)^\T\right\}\right]\gamma(\rd\bmT)
\end{equation}
where $C(p,n)$ is a normalizing constant given by
\begin{align}
 C(p,n)=2^{p(n-2)/2}\pi^{p(p-1)/4}\prod_{i=1}^p\Gamma(\{n+1-i\}/2)
\end{align}
and $\gamma(\rd\bmT)$ is the left-invariant Haar measure on $\mathcal{T}^+$ given by
\begin{align}\label{gammagamma}
 \gamma(\rd\bmT)=\prod_{i=1}^p t_{ii}^{-i}\rd \bmT.
\end{align}
An estimator $\bmde$ is evaluated by the invariant 
loss function given by
\begin{align}\label{general.loss.3}
L( \bmTh\bmde\bmTh^\T).
\end{align}
Denote the risk function by
\begin{align*}
 R(\bmde,\bmSi)=\int_{\mathcal{T}^+}L(\bmTh\bmde\bmTh^\T)f_W(\bmT|\bmTh)\gamma(\rd\bmT).
\end{align*}
For all $\bmA\in\mathcal{T}^+$,
the group transformation with respect to $\mathcal{T}^+$ on a random matrix
$\bmT$ and a parameter matrix $\bmTh$ is defined by $(\bmT,\bmTh)\to (\bmA\bmT,\bmTh\bmA^{-1})$.
The group $\bar{G}$ operating on $\bmTh$ is transitive.
Any equivariant estimator of
\begin{align*}
 \bmSi=(\bmTh^\T \bmTh)^{-1}=\bmTh^{-1}(\bmTh^{-1})^\T
\end{align*}
under the lower triangular group is of form given by
\begin{equation}\label{equiv}
 \bmde(\bmA\bmT)=\bmA\bmde(\bmT)\bmA^\T.
\end{equation}

\begin{thm}\label{thm.BE.matrix}
 Let $\bmV=\bmT\bmT^\T\sim \mathcal{W}_p(n,\bmSi)$ and let the loss be $L(\bmTh\bmde\bmTh^\T)$ as in
 \eqref{general.loss.3}.
Then the generalized Bayes estimator with respect to the prior 
\begin{align}
\pi(\bmTh)=\gamma(\rd\bmTh),
\end{align}
$\bmde_0$, is best equivariant under lower triangular group, that is,
\begin{align}
 \bmde_0(\bmT)=\argmin_{\bmde}\int_{\mathcal{T}^+}L(\bmTh\bmde\bmTh^\T)f_W(\bmT|\bmTh)
 \gamma(\rd\bmTh).
\end{align}
\end{thm}
Note that $\gamma(\rd\bmTh)$ is the ``left'' invariant measure, which seems to contradict
the general theory by \cite{Hora-Buehler-1966}.
However this seeming anomaly is due to our parameterization
$\bmV=\bmT\bmT^\T$, $\bmSi^{-1}=\bmTh^\T \bmTh$ and
\begin{align}
 \bmSi=\bmTh^{-1}(\bmTh^{-1})^\T.
\end{align}
The general theory implies that
\begin{align}
 \nu(\rd\bmTh^{-1})= \gamma(\rd\bmTh)
\end{align}
where $\nu$ is right invariant Haar measure on $\mathcal{T}^+$ given by
\begin{align}
 \nu(\rd\bmZ)=\prod_{i=1}^p z_{ii}^{-(p-i+1)}\rd\bmZ.
\end{align}
In the proof below, in addition to the left invariance of $\gamma$, and the right invariance
of $\nu$, we use the fact that
\begin{equation}\label{formform}
 f_W(\bmT|\bmTh)=f_W(\bmTh \bmT|\bmI)=f_W(\bmI|\bmTh \bmT).
\end{equation}
\begin{proof}[Proof of Theorem \ref{thm.BE.matrix}]
By \eqref{density.W} and \eqref{general.loss.3}, the risk of an equivariant estimator
can be expressed as
\begin{align*}
 & R(\bmde,\bmSi) \\
 &=\int_{\mathcal{T}^+}L(\bmTh\bmde(\bmT)\bmTh^\T)f_W(\bmT|\bmTh)\gamma(\rd\bmT) \\
 &=\int_{\mathcal{T}^+}L(\bmde(\bmTh\bmT))f_W(\bmT|\bmTh)\gamma(\rd\bmT) \\
 &=\int_{\mathcal{T}^+}L(\bmde(\bmZ))f_W(\bmZ|\bmI)\gamma(\rd\bmZ) \quad (\bmZ=\bmTh\bmT, \text{ and left invariance of }\gamma) \\
 &=\int_{\mathcal{T}^+}L(\bmde(\bmZ))f_W(\bmI|\bmZ)\prod_{i=1}^p z_{ii}^{-i}\rd\bmZ \ (\text{by the form of }f_W)\\
 &=\int_{\mathcal{T}^+}L(\bmde(\bmZ))f_W(\bmI|\bmZ)\prod_{i=1}^p z_{ii}^{p-2i+1}
 \nu(\rd\bmZ) \\
 &=\int_{\mathcal{T}^+}L(\bmde(\bmW\bmS))f_W(\bmS|\bmW)
 \prod_{i=1}^p (w_{ii}s_{ii})^{p-2i+1}
\nu(\rd\bmW) \\ &\qquad (\bmZ= \bmW\bmS, \text{ and right invariance of }\nu)\\
 &=\prod_{i=1}^p s_{ii}^{p-2i+1}
 \int_{\mathcal{T}^+}L(\bmW\bmde(\bmS)\bmW^\T)f_W(\bmS|\bmW) \gamma(\rd\bmW), 
 (\text{by \eqref{equiv} and the form of }\gamma(\rd w)) 
\end{align*}
Then the best equivariant estimator with respect to the group $\mathcal{T}^+$
 can be written by 
\begin{align*}
 \bmde_0(\bmT)=\argmin_{\bmde}\int_{\mathcal{T}^+}L(\bmTh\bmde\bmTh^\T)f_W(\bmT|\bmTh)
 \gamma(\rd\bmTh).
\end{align*}
\end{proof}
 \section{Minimaxity}
\label{sec:minimaxity}
In this section, we choose an appropriate sequence of priors whose support is the entirety of the
parameter space and show that the Bayes risks converge to the constant risk of
the best equivariant estimator $\delta_0$.
By a well-known standard result (see e.g.~\cite{Lehmann-Casella-1998}),
this implies minimaxity of $\delta_0$.
In order to deal with explicit expressions for minimax estimators
as well as for somewhat technical reasons,
in this section, we specify the loss functions to be standard choices in the literature.
For the location and scale problem, the squared error loss and the entropy loss
\begin{align*}
 L(\delta-\mu)=(\delta-\mu)^2, \quad L(\delta/\sigma)=\delta/\sigma-\log(\delta/\sigma)-1
\end{align*}
are used respectively. For estimation of covariance matrix,
the so called \citeapos{Stein-1956-tr} loss function given by
	   \begin{equation}\label{Stein-loss.0}
\begin{split}
    L(\bmTh\bmde\bmTh^\T) 
 &=\mathrm{tr}\bmSi^{-1}\bmde-\log|\bmSi^{-1}\bmde|-p\\
 &=\mathrm{tr}(\bmTh\bmde\bmTh^\T) -\log|\bmTh\bmde\bmTh^\T|-p  
\end{split}	
	   \end{equation}
is used.	   

\subsection{Estimation of location}
\label{sec:minimaxity_location}
In this section, we show the minimaxity of $\hat{\mu}^0$, 
the best location equivariant estimator under squared error loss.
A point of departure from most proofs in the literature is that
a smooth sequence of Gaussian densities simplifies the proof.
It is also easily applied in the multivariate location family (See Remark \ref{rem:multi}).

Recall that the Bayes estimator corresponding to a (generalized) prior $\pi(\mu)$,
under squared error loss, is given by
\begin{align}
 \delta_\pi(\bmx)
 &=\argmin_{\delta}\int_{-\infty}^\infty L(\delta-\mu)f(\bmx-\mu)
 \pi(\mu)\rd \mu \label{sq.1}\\
 & =\frac{\int \mu f(\bmx-\mu)\pi(\mu)\rd\mu}{\int f(\bmx-\mu)\pi(\mu)\rd\mu}
\quad\text{under }L(t)=t^2.
 \label{sq.2}
\end{align} 
Hence, by Theorem \ref{thm:BE_location}, the best equivariant estimator is given by
\begin{align}\label{location.best.0}
 \hat{\mu}_0(\bmx)
 =\frac{\int \mu f(\bmx-\mu)\rd\mu}{\int f(\bmx-\mu)\rd\mu}.
\end{align} 
\begin{thm}\label{thm:location.simple}
 Let $\bmX$ have distribution \eqref{location.density} and let the loss
 be given by $L(\delta-\mu)=(\delta-\mu)^2$.
 Then the best equivariant estimator, $\hat{\mu}_0(\bmx)$,   given by \eqref{location.best.0},
 is minimax, and the minimax constant risk is given by
\begin{align*}
R_0 =\int L(\hat{\mu}_0(\bmx))f(\bmx)\rd\bmx
 =\int \left\{\hat{\mu}_0(\bmx)\right\}^2f(\bmx)\rd\bmx.
\end{align*}
\end{thm}
Under the squared error loss, the Bayes estimator is explicitly written as \eqref{sq.2}, 
However, in the following proof, the implicit expression \eqref{sq.1} is mainly used
to indicate possible extension for more general loss functions.
For the same reason, $L(\delta(\bmx)-\mu)$ instead of $(\delta(\bmx)-\mu)^2$ is used.
\begin{proof}[Proof of Theorem \ref{thm:location.simple}]
 Let
 \begin{align*}
  \phi(\mu)=\frac{1}{\sqrt{2\pi}}\exp(-\mu^2/2) 
  \text{ and }
  \phi_k(\mu)=\frac{1}{k}\phi(\mu/k).
 \end{align*}
 The Bayes risk of $\delta(\bmx)$ under the prior $\phi_k(\mu)$ is given by
\begin{align*}
 r_k(\phi_k,\delta(\bmx))=
\iint L(\delta(\bmx)-\mu)f(\bmx-\mu)\phi_k(\mu)\rd\mu\rd\bmx.
\end{align*}
Also the corresponding Bayes estimator is given by
\begin{align*}
 \delta^\phi_k(\bmx)
=\argmin_{\delta}\int_{-\infty}^\infty L(\delta-\mu)f(\bmx-\mu)
 \phi_k(\mu)\rd \mu.
\end{align*} 
 Clearly 
\begin{align*}
 r_k(\phi_k,\delta^\phi_k)\leq r_k(\phi_k,\hat{\mu}_0)=R_0,
\end{align*}
 and therefore, to show $  \lim_{k\to\infty} r_k(\phi_k,\delta^\phi_k)= R_0$,
it suffices  to prove
\begin{align*}
 \liminf_{k\to\infty} r_k(\phi_k,\delta^\phi_k)\geq R_0.
\end{align*}
Making the transformation $\bmz=\bmx-\mu$ yields
\begin{align*}
 r_k(\phi_k,\delta^\phi_k)=
\iint
 L(\delta^\phi_k(\bmz+\mu)-\mu)f(\bmz)\phi_k(\mu)\rd\mu\rd\bmz
\end{align*}
 where
\begin{align*}
 \delta^\phi_k(\bmz+\mu)
 = \argmin_{\delta}\int_{-\infty}^\infty L(\delta-\theta)f(\bmz+\mu-\theta)
 \phi_k(\theta)\rd \theta.
\end{align*}
 Now, make the transformation $t=\theta-\mu$.
We then have
\begin{align*}
 \delta^\phi_k(\bmz+\mu)
 = \argmin_{\delta}\int_{-\infty}^\infty L(\delta-\mu-t)f(\bmz-t)
 \phi_k(t+\mu)\rd t
\end{align*}
or equivalently
\begin{align*}
\delta^*_k(\bmz,\mu):= \delta^\phi_k(\bmz+\mu)-\mu=
 \argmin_{\delta}\int_{-\infty}^\infty L(\delta-t)f(\bmz-t)
 \phi_k(\mu+t)\rd t .
\end{align*}
 Hence, by change of variables, we have
\begin{align*}
 r_k(\phi_k,\delta^\phi_k)&=
\iint
 L(\delta^*_k(\bmz,\mu))f(\bmz)\phi_k(\mu)\rd\mu\rd\bmz \\
 &=
\iint
 L(\delta^*_k(\bmz,k\mu))f(\bmz)\phi(\mu)\rd\mu\rd\bmz. 
\end{align*}
 Note also $k\phi_k(t+k\mu)=\phi(t/k+\mu)$ and
\begin{align*}
\delta^*_k(\bmz,k\mu)& =\argmin_{\delta}\int_{-\infty}^\infty L(\delta-t)f(\bmz-t)
 k\phi_k(k\mu+t)\rd t \\ &=
\frac{\int_{-\infty}^\infty t f(\bmz-t)
 \phi_k(t/k+\mu)\rd t}{\int_{-\infty}^\infty f(\bmz-t)
 \phi_k(t/k+\mu)\rd t} \ (\text{for squared error loss }L(t)=t^2).
\end{align*}
Since $\lim_{k\to\infty}\phi(t/k+\mu)=\phi(\mu)$ 
for any $\mu$, the dominated convergence theorem implies
 \begin{equation}\label{important.1}
  \lim_{k\to\infty}  \delta^*_k(\bmz,k\mu)=\hat{\mu}_0(\bmz)
 \end{equation}
 and hence
 \begin{align}\label{important.2}
  \lim_{k\to\infty} L(\delta^*_k(\bmz,k\mu))=\lim_{k\to\infty}
  \{\delta^*_k(\bmz,k\mu)\}^2=\{\hat{\mu}_0(\bmz)\}^2=L(\hat{\mu}_0(\bmz)).
 \end{align}
 Hence by 
 Fatou's lemma, we obtain that
\begin{equation}\label{fatou.2.location}
 \begin{split}
\liminf_{k\to\infty}   r_k(\phi_k,\delta^\phi_k)
&= \liminf_{k\to\infty}\iint
 L(\delta^*_k(\bmz,k\mu))f(\bmz)\phi(\mu)\rd\mu\rd \bmz \\
 &\geq
 \iint
 \liminf_{k\to\infty}L(\delta^*_k(\bmz,k\mu))f(\bmz)\phi(\mu)\rd\mu\rd\bmz \\
 &=
 \iint
 L(\hat{\mu}_0(\bmz))f(\bmz)\phi(\mu)\rd\mu\rd\bmz \\
  &= R_0.
\end{split}
\end{equation}
\end{proof}

\begin{remark}\label{rem:multi}
In the multivariate case, suppose $ \bmx_1,\dots,\bmx_p\in\mathbb{R}^n$ and
\begin{align*}
 \{\bmx_1,\dots,\bmx_p\}\sim f(\bmx_1-\mu_1,\dots,\bmx_p-\mu_p).
\end{align*}
 Let $\bmmu=(\mu_1,\dots,\mu_p)^\T$. Then the Pitman estimator of $\bmmu$,
 the generalized Bayes estimator with respect to $\pi(\bmmu)=1$, is
\begin{equation}\label{multivariate-location}
 \hat{\bmmu}(\bmx_1,\dots,\bmx_p)
 =\frac{\int_{\mathbb{R}^p}\bmmu f(\bmx_1-\mu_1,\dots,\bmx_p-\mu_p)\rd\bmmu}
 {\int_{\mathbb{R}^p} f(\bmx_1-\mu_1,\dots,\bmx_p-\mu_p)\rd\bmmu}.
\end{equation}
Using
\begin{align*}
 \pi_k(\bmmu)=\prod_{i=1}^p\phi_k(\mu_i)=\frac{1}{(2\pi k^2)^{p/2}}
 \exp\left(-\frac{\|\bmmu\|^2}{2k^2}\right)
\end{align*}
as the least favorable sequence of priors gives minimaxity
under the quadratic loss $\|\bmde-\bmmu\|^2$ of \eqref{multivariate-location}.
\end{remark}

 \subsection{Estimation of scale}
\label{sec:minimaxity_scale}
In this section, we show the minimaxity of the scale Pitman estimator
under entropy loss given by
\begin{equation}\label{entropy.entropy}
 L(\delta/\sigma)=\delta/\sigma-\log(\delta/\sigma)-1.
\end{equation}
Recall that the Bayes estimator corresponding to a (generalized) prior $\pi(\sigma)$,
under entropy loss \eqref{entropy.entropy}, is given by
\begin{align}
 \delta_\pi(\bmx)
 &=\argmin_{\delta}\int_{-\infty}^\infty L(\delta/\sigma)\sigma^{-n}f(\bmx/\sigma)
 \pi(\sigma)\rd \sigma \label{ent.1}\\
 & =\frac{\int  \sigma^{-n}f(\bmx/\sigma)\pi(\sigma)\rd\sigma}
{\int  \sigma^{-n-1}f(\bmx/\sigma)\pi(\sigma)\rd\sigma}.\label{ent.2}
\end{align} 
Hence the generalized Bayes estimator under $\pi(\sigma)=1/\sigma$,
which is best equivariant as shown in Theorem \ref{thm:be_s}, is given by
\begin{align}\label{scale.best.0}
 \hat{\sigma}_0(\bmx)
=\frac{\int  \sigma^{-n-1}f(\bmx/\sigma)\rd\sigma}
{\int  \sigma^{-n-2}f(\bmx/\sigma)\rd\sigma}.
\end{align} 
We have a following minimaxity result.
\begin{thm}\label{thm:scale}
 Let $\bmX$ have distribution \eqref{scale.density} and let the loss
 be given by $L(\delta/\sigma)=\delta/\sigma-\log(\delta/\sigma)-1$.
 Then the best equivariant estimator, $\hat{\sigma}_0(\bmx)$,   given by \eqref{scale.best.0},
 is minimax, and the minimax constant risk is given by
\begin{align*}
R_0 =\int L(\hat{\sigma}_0(\bmx))f(\bmx)\rd\bmx
 =\int \left\{\hat{\sigma}_0(\bmx)-\log\hat{\sigma}_0(\bmx)-1\right\}f(\bmx)\rd\bmx.
\end{align*}
\end{thm}

\begin{proof}
Assume $\log\sigma \sim N(0,k^2)$
or equivalently
\begin{align*}
\pi_k(\sigma)=\frac{1}{k}\phi(\log \sigma/k)\frac{1}{\sigma},
\end{align*}
 where $\phi(\cdot)$ is the pdf of $N(0,1)$.
Then the Bayes estimator satisfies 
\begin{align*}
 \delta^\pi_k=\delta^\pi_k(\bmx)=
 \argmin_{\delta}
\int_0^\infty L(\delta/\sigma)\sigma^{-n}f(\bmx/\sigma)\phi_k(\sigma)\rd\sigma
\end{align*}
and the Bayes risk is given by
\begin{align*}
 r_k(\pi_k,\delta^\pi_k)=\iint
 L(\delta/\sigma)\sigma^{-n}f(\bmx/\sigma)\pi_k(\sigma)\rd\sigma\rd\bmx.
\end{align*}

 Clearly 
\begin{align*}
 r_k(\pi_k,\delta^\pi_k)\leq r_k(\pi_k,\hat{\sigma}_0(\bmx))=R_0,
\end{align*}
 and therefore, to show $  \lim_{k\to\infty} r_k(\phi_k,\delta^\phi_k)= R_0$,
it suffices  to prove
\begin{align*}
 \liminf_{k\to\infty} r_k(\pi_k,\delta^\pi_k)\geq R_0.
\end{align*}
Making the transformation $\bml=\bmx/\sigma$ yields
\begin{align*}
 r_k(\pi_k,\delta^\pi_k)=
\iint
 L(\delta^\pi_k(\sigma\bml)/\sigma)f(\bml)\pi_k(\sigma)\rd\sigma\rd\bml
\end{align*}
 where
\begin{align*}
 \delta^\pi_k(\sigma\bml)
 = \argmin_{\delta}\int_0^\infty L(\delta/z)z^{-n}f(\sigma\bml/z)
 \pi_k(z)\rd z.
\end{align*}
 Now, make the transformation $y=z/\sigma$.
We then have
\begin{align*}
 \delta^\pi_k(\sigma\bml)
 = \argmin_{\delta}\int_0^\infty L(\delta/(y\sigma))y^{-n}f(\bml/y)
 \pi_k(\sigma y)\rd y 
\end{align*}
or equivalently
\begin{align*}
\delta^*_k(\bml,\sigma):= \frac{\delta^\pi_k(\sigma\bml)}{\sigma}=
 \argmin_{\delta}\int_0^\infty L(\delta/y)y^{-n}f(\bml/y)
 \pi_k(\sigma y)\rd y .
\end{align*}
 Hence
\begin{align*}
 r_k(\pi_k,\delta^\pi_k)&=\iint L(\delta^*_k(\bml,\sigma))f(\bml)\pi_k(\sigma)\rd\sigma\rd\bml \\
 &=\iint L(\delta^*_k(\bml,\eta^k))f(\bml)\pi_1(\eta)\rd\eta\rd\bml 
\end{align*}
 where $\sigma=\eta^k$ and $\delta^*_k(\bml,\eta^k)$ is explicitly given as
 (when the loss is \eqref{entropy.entropy})
 \begin{equation}
  \delta^*_k(\bml,\eta^k) =\frac{\int y^{-n}f(\bml/y) \pi_k(\eta^k y)\rd y }
   {\int y^{-n-1}f(\bml/y) \pi_k(\eta^k y)\rd y} . 
 \end{equation}
Note
\begin{align*}
 k \pi_k(\eta^k y)=\frac{1}{\eta^k y}\frac{1}{\sqrt{2\pi}}\exp\left(-\frac{(\log \eta^k +\log y)^2}{2k^2}\right).
\end{align*}
Since 
 \begin{align*}
\lim_{k\to\infty} k \eta^k \pi_k(\eta^k y)=\frac{1}{y}\phi(\log \eta)
 \end{align*}
for any $\eta$,
the dominated convergence theorem implies
\begin{equation}
\lim_{k\to\infty} \delta^*_k(\bml,\eta^k)=\hat{\sigma}_0(\bml).
\end{equation}
Also the continuity of $L(\cdot)$ implies
 \begin{equation}
  \lim_{k\to\infty}L(\delta^*_k(\bml,\eta^k))=L(\hat{\sigma}_0(\bml)).
 \end{equation}
 Hence by   Fatou's lemma, we obtain that
\begin{equation}\label{fatou.2}
 \begin{split}
\liminf_{k\to\infty}   r_k(\pi_k,\delta^\pi_k)
&= \liminf_{k\to\infty}\iint
 L(\delta^*_k(\bml,\eta^k))f(\bml)\pi_1(\eta)\rd\eta\rd\bml \\
 &\geq
 \iint
 \liminf_{k\to\infty}L(\delta^*_k(\bml,\eta^k))f(\bml)\pi_1(\eta)\rd\eta\rd\bml \\
 &\geq
 \iint
 L(\hat{\sigma}_0(\bml))f(\bml)\pi_1(\eta)\rd\eta\rd\bml \\
  &=  R_0.
\end{split}
\end{equation}
\end{proof}

\begin{remark}
 In the same way, we can consider the estimation of $\sigma^c$ with $c\in\mathbb{R}$
 and propose the corresponding result,
\begin{align*}
 \hat{\sigma}_{0c}(\bmx)
=\frac{\int  \sigma^{-n-1+c}f(\bmx/\sigma)\rd\sigma}
{\int  \sigma^{-n-2+c}f(\bmx/\sigma)\rd\sigma}
\end{align*} 
 is minimax and best equivariant for estimating $\sigma^c$ under entropy loss
 \begin{align*}
 L(\delta/\sigma^c)=\delta/\sigma^c-\log(\delta/\sigma^c)-1.  
 \end{align*}
\end{remark}

\subsection{Estimation of covariance matrix}
\label{sec:minimaxity_matrix}
As we mentioned in the beginning of this section, we use 
the so called \citeapos{Stein-1956-tr} loss function given by
	   \begin{equation}\label{Stein-loss.00}
\begin{split}
    L(\bmTh\bmde\bmTh^\T) 
 &=\mathrm{tr}\bmSi^{-1}\bmde-\log|\bmSi^{-1}\bmde|-p\\
 &=\mathrm{tr}(\bmTh\bmde\bmTh^\T) -\log|\bmTh\bmde\bmTh^\T|-p  .
\end{split}	
	   \end{equation}
\cite{James-Stein-1961}, in their Section 5,
show that the best equivariant estimator is given by
\begin{align}\label{JSJS}
 \hat{\bmSi}_0=\bmT\mathrm{diag}(d_1,\dots,d_p)\bmT^\T
\end{align}
where $\bmT\in\mathcal{T}^+$ is from the Cholesky decomposition of $\bmV=\bmT\bmT^\T$ and
$d_i=1/(n+p-2i+1)$ for $i=1,\dots,p$.
As demonstrated in the literature, by e.g.~\cite{Tsukuma-Kubokawa-2015},
the best equivariant estimator under the loss \eqref{Stein-loss.00}
may also be shown to be $\hat{\bmSi}_0$ by using the generalized Bayes
representation given in Theorem \ref{thm.BE.matrix} since
\begin{align*}
& \argmin_{\bmde}\int_{\mathcal{T}^+}L(\bmTh\bmde\bmTh^\T)f_W(\bmT|\bmTh)
 \gamma(\rd\bmTh) \\
& \argmin_{\bmde}\int_{\mathcal{T}^+}\left\{\mathrm{tr}\left(\bmTh^\T\bmTh\bmde\right)-\log|\bmde|\right\}f_W(\bmT|\bmTh)
 \gamma(\rd\bmTh) \\
 &=
\left(\int_{\mathcal{T}^+}\bmTh^\T\bmTh f_W(\bmT|\bmTh) \gamma(\rd\bmTh)\right)^{-1}
 \int_{\mathcal{T}^+} f_W(\bmT|\bmTh) \gamma(\rd\bmTh) \\
 &=
\left(\int_{\mathcal{T}^+}\bmTh^\T\bmTh f_W(\bmT|\bmTh) \prod \theta_{ii}^{p-2i+1}\nu(\rd\bmTh)\right)^{-1}\\
&\qquad\times \int_{\mathcal{T}^+} f_W(\bmT|\bmTh) \prod \theta_{ii}^{p-2i+1}\nu(\rd\bmTh) \\
 &=
\left(\int_{\mathcal{T}^+}(\bmZ\bmT^{-1})^\T\bmZ\bmT^{-1} f_W(\bmZ|\bmI) \prod (t^{-1}_{ii}z_{ii}^{p-2i+1})\nu(\rd\bmZ)\right)^{-1}\\
&\qquad\times \int_{\mathcal{T}^+} f_W(\bmZ|\bmI) \prod (t^{-1}_{ii}z_{ii}^{p-2i+1})\nu(\rd\bmZ) \\
 &=
\bmT\left(\int_{\mathcal{T}^+}\bmZ^\T\bmZ f_W(\bmZ|\bmI) \prod z_{ii}^{p-2i+1}\nu(\rd\bmZ)\right)^{-1}\bmT^\T\\
&\qquad\times \int_{\mathcal{T}^+} f_W(\bmZ|\bmI) \prod z_{ii}^{p-2i+1}\nu(\rd\bmZ) \\
&=\hat{\bmSi}_0
\end{align*}
where
\begin{align*}
& \left(\int_{\mathcal{T}^+}\bmZ^\T\bmZ f_W(\bmZ|\bmI) \prod z_{ii}^{p-2i+1}\nu(\rd\bmZ)\right)^{-1}
\int_{\mathcal{T}^+} f_W(\bmZ|\bmI) \prod z_{ii}^{p-2i+1}\nu(\rd\bmZ) \\
&=\mathrm{diag}(d_1,\dots,d_p) \text{ with }d_i=1/(n+p-2i+1).
\end{align*}
Note that the group of $p\times p$ lower triangular matrices with
positive diagonal entries is solvable, and the result of \cite{Kiefer-1957} implies the minimaxity of
$ \hat{\bmSi}_0$.
\cite{Tsukuma-Kubokawa-2015} gives as a sequence of least favorable priors,
a sequence of invariant priors truncated on an expanding set.

In this section, 
we choose an appropriate sequence of Gaussian priors whose support is the entirety of the
parameter space and show that the Bayes risks converge to the constant risk of $\hat{\bmSi}_0$.
This implies that $\hat{\bmSi}_0$ is minimax.

%
As a new parameterization on $\bmTh$,
let
\begin{equation}\label{intro.xi}
 \begin{split}
 \xi_{ii}&=\log\theta_{ii}\text{ for }i=1,\dots,p, \\
 \xi_{ij}&=\frac{\theta_{ij}}{\theta_{ii}},\text{ for }1\leq j< i\leq p
\end{split}
\end{equation}
and let
\begin{equation}
 \bmxi=(\xi_{11},\xi_{21},\xi_{22},\dots,\xi_{p1},\dots,\xi_{pp})^\T\in\mathbb{R}^{p(p+1)/2}.
\end{equation}
The prior on $\xi_{ij}$ is
\begin{align*}
 \xi_{ij}\sim N(0,k_{ij}^2)\text{ for }1\leq j\leq i\leq p,
\end{align*}
and $\xi_{11},\xi_{21},\xi_{22},\dots,\xi_{p1},\dots,\xi_{pp}$ are assumed mutually independent.
Equivalently the density is 
\begin{align}\label{prior.xi}
 \bar{\pi}_k(\bmxi)\rd\bmxi
=\prod_{j\leq i}k_{ij}^{-1}\phi(\xi_{ij}/k_{ij})
\rd \bmxi
\end{align}
where $\phi(t)=(2\pi)^{-1/2}\exp(-t^2/2)$. 
Set
\begin{equation}\label{assumption.k}
 k_{ii}=k, \quad k_{ij}=k^{(i-j)k} \text{ for }i>j
\end{equation}
with $k\to\infty$, although, in the following, we keep the notation $k_{ii}$ and $k_{ij}$.

By \eqref{intro.xi} and \eqref{prior.xi}, we have
\begin{equation}\label{prior.Th}
 \pi_k (\bmTh)\rd\bmTh
 =\prod_{i=1}^p
 \left\{\frac{1}{k_{ii}}\phi(\log\theta_{ii}/k_{ii})\frac{1}{\theta_{ii}}\right\}
\prod_{j<i}\left\{\frac{1}{k_{ij}\theta_{ii}}\phi(\theta_{ij}/\{\theta_{ii}k_{ij}\})\right\}
\rd \bmTh.
\end{equation}
The prior distributions yield the Bayes estimators
\begin{align*}
 \bmde^\pi_k =\bmde^\pi_k (\bmT)&=
 \argmin_{\bmde}\int_{\bmZ\in\mathcal{T}^+} L(\bmZ\bmde\bmZ^\T)f_W(\bmT|\bmZ)\pi_k (\bmZ)\rd\bmZ
\end{align*}
with Bayes risks
\begin{equation}\label{br.3.1}
 r_k (\pi_k ,\bmde^\pi_k )=\iint
 L(\bmTh\bmde^\pi_k (\bmT)\bmTh^\T)f_W(\bmT|\bmTh)\gamma(\rd\bmT)\pi_k (\bmTh)\rd\bmTh.
\end{equation}

\begin{thm}
 Let $\bmV=\bmT\bmT^\T$ have distribution $\mathcal{W}_p(n,\bmSi)$ and let the loss
 be given by \eqref{Stein-loss.00}.
 Then the best equivariant estimator, $\hat{\bmSi}_0(\bmT)$,  given by \eqref{JSJS},
 is minimax, and the minimax constant risk is given by
\begin{align*}
R_0 =\int L(\hat{\Sigma}_0(\bmT))f_W(\bmT|\bmI)\gamma(\rd \bmT).
\end{align*}
\end{thm}
\begin{proof}
We show this theorem along the same lines as in \cite{Kubokawa-2004} and
\cite{Tsukuma-Kubokawa-2015} who modified the method of \cite{Girshick-Savage-1951}.

 Clearly 
\begin{align*}
 r_k (\pi_k ,\bmde^\pi_k )\leq r_k (\pi_k ,\hat{\bmSi}_0)=R_0,
\end{align*}
 and therefore, to show $  \lim_{k\to\infty} r_k(\phi_k,\delta^\phi_k)= R_0$,
it suffices  to prove
\begin{align*}
 \liminf_{k\to\infty} r_k (\pi_k ,\bmde^\pi_k )\geq R_0.
\end{align*}
In \eqref{br.3.1}, making the transformation $\bmL=\bmTh\bmT$ yields
\begin{equation}\label{br.3.2}
 r_k (\pi_k ,\bmde^\pi_k )=\iint
 L(\bmTh\bmde^\pi_k (\bmTh^{-1}\bmL)\bmTh^\T)f_W(\bmL|\bmI_p)\gamma(\rd\bmL)\pi_k (\bmTh)
 \rd\bmTh
\end{equation}
 where
\begin{align*}
 \bmde^\pi_k (\bmTh^{-1}\bmL)
 = \argmin_{\bmde}\int_{\bmZ\in\mathcal{T}^+} L(\bmZ\bmde\bmZ^\T)f_W(\bmL|\bmZ\bmTh^{-1})
 \pi_k (\bmZ)\rd\bmZ.
\end{align*}
 Now, make the transformation $\bmY=\bmZ\bmTh^{-1}$ with $\rd\bmZ=(\prod_{i=1}^p \theta_{ii}^{p-i+1})\rd\bmY$.
We then have
\begin{align*}
 \bmde^\pi_k (\bmTh^{-1}\bmL)
 = \argmin_{\bmde}\int_{\bmY\in\mathcal{T}^+} L(\bmY\bmTh\bmde\bmTh^\T\bmY^\T)f_W(\bmL|\bmY)
 \pi_k (\bmY\bmTh)\rd\bmY
\end{align*}
 namely,
 \begin{align*}
  \bmTh\bmde^\pi_k (\bmTh^{-1}\bmL)\bmTh^\T=\bmde_k ^*(\bmL|\bmTh)
 \end{align*}
 where
 \begin{align*}
\bmde_k ^*(\bmL|\bmTh)=  \argmin_{\bmde}\int_{\bmY\in\mathcal{T}^+} L(\bmY\bmde\bmY^\T)f_W(\bmL|\bmY)
 \pi_k (\bmY\bmTh)\rd\bmY.
 \end{align*}
 Hence, the Bayes risk \eqref{br.3.2} can be rewritten as
\begin{equation}\label{br.3.3}
 \begin{split}
 r_k (\pi_k ,\bmde^\pi_k )&=\iint
 L(\bmde^*_k (\bmL|\bmTh))f_W(\bmL|\bmI_p)\gamma(\rd\bmL)\pi_k (\bmTh)
  \rd\bmTh \\
  &=\iint
 L(\bmde^*_k (\bmL|\bmTh(\bmxi)))f_W(\bmL|\bmI_p)\gamma(\rd\bmL)\bar{\pi}_k (\bmxi)
  \rd\bmxi,
   \end{split}
\end{equation} 
 where $\bmTh(\bmxi)$ is from $\xi_{ii}=\log\theta_{ii}$ for $i=1,\dots,p$ and
 $\xi_{ij}=\theta_{ij}/\theta_{ii}$ for $1\leq j<i\leq p$.
Then we have 
 \begin{align*}
 r_k (\pi_k ,\bmde^\pi_k )
  &= \iint
 L(\bmde^*_k (\bmL|\bmTh(\bmxi)))f_W(\bmL|\bmI_p)\gamma(\rd\bmL)\bar{\pi}_k (\bmxi)
  \rd\bmxi \\
  &=\iint
  L(\bmde^*_k (\bmL|\bmTh(\bmk\!\bullet\!\bm{\omega})))f_W(\bmL|\bmI_p)\gamma(\rd\bmL)\bar{\pi}_1(\bm{\omega})
  \rd\bm{\omega},
 \end{align*}
where, for notational convenience,
 \begin{align*}
 \bmk\!\bullet\!\bm{\omega}&=(k_{11}\omega_{11},k_{21}\omega_{21},k_{22}\omega_{22},\dots,k_{pp}\omega_{pp}) \\
 \bmTh(\bmk\!\bullet\!\bm{\omega})_{ii}&=\exp(k_{ii}\omega_{ii}) \text{ for }i=1,\dots,p, \\
 \bmTh(\bmk\!\bullet\!\bm{\omega})_{ij}&=k_{ij}\omega_{ij}\exp(k_{ii}\omega_{ii}),
 \text{ for }1\leq j<i\leq p.
\end{align*}

By Lemma \ref{saigo} below, we have
\begin{equation}
 \lim_{k\to \infty} \bmde^*_k (\bmL|\bmTh(\bmk\!\bullet\!\bm{\omega}))=\hat{\bmSi}_0(\bmL)
\end{equation}
 and by the continuity of $L(\cdot)$,
\begin{align}
\lim_{k\to\infty}L(\bmde^*_k (\bmL|\bmTh(\bmk\!\bullet\!\bm{\omega})))=L(\hat{\bmSi}_0(\bmL)).
\end{align} 
Also, by Fatou's lemma, we have
\begin{align*}
&\liminf_{k\to\infty} r_k (\pi_k ,\bmde^\pi_k ) \\
 &\geq \liminf_{k\to\infty}
 \int\int_{\bmL\in\mathcal{T}^+}
  L(\bmde^*_k (\bmL|\bmTh(\bmk\!\bullet\!\bm{\omega})))f_W(\bmL|\bmI_p)\gamma(\rd\bmL)\bar{\pi}_1(\bm{\omega})
  \rd\bm{\omega}\\
 &\geq 
 \int
 \int_{\bmL\in\mathcal{T}^+}
  \liminf_{k\to\infty}L(\bmde^*_k (\bmL|\bmTh(\bmk\!\bullet\!\bm{\omega})))f_W(\bmL|\bmI_p)\gamma(\rd\bmL)\bar{\pi}_1(\bm{\omega})
  \rd\bm{\omega} \\
 &=
 \int
 \int_{\bmL\in\mathcal{T}^+}
  L(\hat{\bmSi}_0(\bmL))f_W(\bmL|\bmI_p)\gamma(\rd\bmL)\bar{\pi}_1(\bm{\omega})
  \rd\bm{\omega} \\
 &=
 \int
 \bar{\pi}_1(\bm{\omega})  \rd\bm{\omega}
  \int_{\bmL\in\mathcal{T}^+}L(\hat{\bmSi}_0(\bmL))f_W(\bmL|\bmI_p)\gamma(\rd\bmL) \\
&=R_0. 
\end{align*}
\end{proof}

\begin{lemma}\label{saigo}
 \begin{align*}
  \lim_{k\to \infty} \bmde^*_k (\bmL|\bmTh(\bmk\!\bullet\!\bm{\omega}))=\hat{\bmSi}_0(\bmL).
 \end{align*}
\end{lemma}
\begin{proof}
Recall
\begin{equation}\label{eq.saigo.lemma}
 \begin{split}
\bmde_k ^*(\bmL|\bmTh)&=  \argmin_{\bmde}\int_{\bmY\in\mathcal{T}^+} L(\bmY\bmde\bmY^\T)f_W(\bmL|\bmY)
 \pi_k (\bmY\bmTh)\rd\bmY \\
&=\left(\int_{\bmY\in\mathcal{T}^+} \bmY^\T\bmY f_W(\bmL|\bmY)
  \pi_k (\bmY\bmTh)\rd\bmY\right)^{-1}\\
&\qquad \times  \int_{\bmY\in\mathcal{T}^+} f_W(\bmL|\bmY)\pi_k (\bmY\bmTh)\rd\bmY
 \end{split}
\end{equation} 
and where
 \begin{align*}
   \pi_k (\bmTh)
 =\prod_{i=1}^p
 \left\{\frac{1}{k_{ii}}\phi(\log\theta_{ii}/k_{ii})\frac{1}{\theta_{ii}}\right\}
\prod_{j<i}\left\{\frac{1}{k_{ij}\theta_{ii}}\phi(\theta_{ij}/\{\theta_{ii}k_{ij}\})\right\}.
 \end{align*}
 Consider $ \pi_k (\bmY\bmTh(\bmk\!\bullet\!\bm{\omega}))$ in the following.
 The $(i,i)$ diagonal component of $\bmY\bmTh(\bmk\!\bullet\!\bm{\omega})$
with $\bmY\in\mathcal{T}^+$ 
 is
\begin{align*}
 y_{ii}\exp(k_{ii}\omega_{ii})
\end{align*}
 and the non-diagonal $(i,j)$ component is
 \begin{align*}
  y_{ij}\exp(k_{jj}\omega_{jj}) + \sum_{l=j+1}^{i-1} y_{il}
\exp(k_{ll}\omega_{ll})k_{lj}\omega_{lj}
  +y_{ii}k_{ij}\omega_{ij}\exp(k_{ii}\omega_{ii}).
 \end{align*}
Then, for $i>j$,
\begin{equation}\label{3.29tabun}
 \begin{split}
 \frac{1}{k_{ij}}
 \frac{(\bmY\bmTh(\bmk\!\bullet\!\bm{\omega}))_{ij}}
 {(\bmY\bmTh(\bmk\!\bullet\!\bm{\omega}))_{ii}}
&=\frac{y_{ij}}{y_{ii}}\frac{\exp(k_{jj}\omega_{jj}-k_{ii}\omega_{ii})}{k_{ij}}+\omega_{ij} \\
 &\qquad +\sum_{l=j+1}^{i-1} w_{lj}\frac{y_{il}}{y_{ii}}
 \frac{k_{lj}}{k_{ij}}\exp(k_{ll}\omega_{ll}-k_{ii}\omega_{ii}).
\end{split}
\end{equation}
Recall we set
\begin{align*}
 k_{ii}=k, \quad k_{ij}=k^{(i-j)k}.
\end{align*}
Then \eqref{3.29tabun} is equal to
 \begin{align*}
& \frac{1}{k_{ij}}
 \frac{(\bmY\bmTh(\bmk\!\bullet\!\bm{\omega}))_{ij}}
 {(\bmY\bmTh(\bmk\!\bullet\!\bm{\omega}))_{ii}} \\
&=  \frac{y_{ij}}{y_{ii}}\left(\frac{\exp(\omega_{jj}-\omega_{ii})}{k^{i-j}}\right)^k
  +\omega_{ij} 
  +\sum_{l=j+1}^{i-1} w_{lj}\frac{y_{il}}{y_{ii}}
\left(\frac{\exp(\omega_{ll}-\omega_{ii})}{k^{i-l}}\right)^k
 \end{align*}
 and hence it follows that
\begin{align*}
\lim_{k\to\infty} \frac{1}{k_{ij}}
 \frac{(\bmY\bmTh(\bmk\!\bullet\!\bm{\omega}))_{ij}}
 {(\bmY\bmTh(\bmk\!\bullet\!\bm{\omega}))_{ii}}=\omega_{ij}.
\end{align*} 
 Similarly we have
 \begin{align*}
 \lim_{k\to\infty} \frac{1}{k_{ii}}
\log (\bmY\bmTh(\bmk\!\bullet\!\bm{\omega}))_{ii}=
\lim_{k\to\infty} \frac{\log y_{ii}+k_{ii}\omega_{ii} }{k_{ii}}
 =\omega_{ii}.
 \end{align*}
Therefore
\begin{align}
 \lim_{k\to\infty}\prod_{i\geq j}k_{ij}\prod_{i=1}^p\exp(ik_{ii}\omega_{ii})
\pi_k (\bmY\bmTh(\bmk\!\bullet\!\bm{\omega}))
=\frac{\prod_{j\geq j}\phi(\omega_{ij})}{\prod_{i=1}^p y_{ii}^i },
\end{align}
and, by the dominated convergence theorem,
\begin{equation*}
 \begin{split}
\bmde_k ^*(\bmL|\bmTh(\bmk\!\bullet\!\bm{\omega}))&=  \argmin_{\bmde}\int_{\bmY\in\mathcal{T}^+} L(\bmY\bmde\bmY^\T)f_W(\bmL|\bmY)
 \pi_k (\bmY\bmTh(\bmk\!\bullet\!\bm{\omega}))\rd\bmY \\
&=\left(\int_{\bmY\in\mathcal{T}^+} \bmY^\T\bmY f_W(\bmL|\bmY)
  \pi_k (\bmY\bmTh(\bmk\!\bullet\!\bm{\omega}))\rd\bmY\right)^{-1}\\
  &\qquad \times  \int_{\bmY\in\mathcal{T}^+} f_W(\bmL|\bmY)\pi_k (\bmY\bmTh(\bmk\!\bullet\!\bm{\omega}))\rd\bmY \\
  &\to
  \left(\int_{\bmY\in\mathcal{T}^+} \bmY^\T\bmY f_W(\bmL|\bmY)\gamma(\rd\bmY)
  \right)^{-1}\\
  &\qquad \times  \int_{\bmY\in\mathcal{T}^+} f_W(\bmL|\bmY)\gamma(\rd\bmY)\rd\bmY \\
  &=\hat{\bmSi}_0(\bmL).
 \end{split}
\end{equation*} 
\end{proof}

\section{Concluding remarks}
\label{sec:CR}
We have reviewed some known results on establishing minimaxity of best equivariant procedures.
While none of the results established are new, the proofs of minimaxity are somewhat divergent
from the typical minimaxity proofs in the literature in that the least favorable sequence
is smooth and strictly positive on the support of the approximated right invariant measure:
it is not a sequence of truncated versions of the invariant prior on expanding sets.
In this sense, our proofs are in the same spirit as the common textbook 
proof of minimaxity of the mean of a normal distribution.
In fact the same sequence of priors that works in the normal case is shown to work
in the general location case. Hence the present method provides a degree of unification
and simultaneously simplifies the proofs.

We note that the Gaussian kernel is not necessary and could be replaced by
a bounded, continuous, positive density.

  Our choices of particular loss featured in each of the problems also simplified
  the analyses, in the sense that, in each case, the form of the Bayes estimate
  could be explicitly given.
  This facilitated the use of Fatou's Lemma in establishing the limiting Bayes risk
  as being equal to the constant risk of the best equivariant estimator.

  An approach for more general loss function could be constructed by requiring
  that all Bayes estimators (for priors with full support) be unique, and that
  the loss is sufficiently smooth that statements such as \eqref{important.1} hold in each
  problem.

  It is also worth noting that, as in \cite{Tsukuma-Kubokawa-2015},
  in the problem of estimating a covariance matrix, a specific least favorable sequence of priors
  is established.

\end{document}